\documentclass[12pt]{amsart}

\usepackage{amsthm}
\usepackage{amssymb}
\usepackage{amsmath}
\usepackage{amscd}
\usepackage{fullpage}
\usepackage[all]{xy}
\usepackage[dvips]{graphicx}
\usepackage{enumitem}

\swapnumbers
\numberwithin{equation}{section}

\newtheorem {Lemma}[equation]       {Lemma}
\newtheorem {lemma}[equation]       {Lemma}
\newtheorem {Theorem}[equation]     {Theorem}

\newtheorem* {Theorem*}     {Theorem}
\numberwithin{equation}{section}

\newtheorem {Proposition}[equation] {Proposition}

\newtheorem {Corollary}[equation]   {Corollary}
\newtheorem*{Lemma*}{Lemma}

\theoremstyle{definition}
\newtheorem {Definition}[equation]  {Definition}

\theoremstyle{remark}
\newtheorem {Example}[equation]     {Example}

\newtheorem {Remark}[equation]      {Remark}
\newtheorem* {Remark*}      {Remark}
\newtheorem {Annoying Remark}[equation]      {Annoying Remark}
\newtheorem* {Example*}          {Example}

\def \SO {{\operatorname{SO}}}
\def \SU {{\operatorname{SU}}}
\def \g {{\mathfrak g}}
\def \h {{\mathfrak h}}
\def \t {{\mathfrak t}}

\def \Z {{\mathbb Z}}
\def \R {{\mathbb R}}
\def \C {{\mathbb C}}
\def \CP {{\mathbb C}{\mathbb P}}

\def \inv {^{-1}}
\def \ssminus {\smallsetminus}
\DeclareMathOperator \Ad {Ad}

\def \princ {\text{princ}}

\def \red {\text{red}}

\def\x{\times}

\def\om{\omega}

\def\sub{\subseteq}

\def \calL {\mathcal{L}}
\def \LL   {\mathcal{L}}

\def \calT {{\mathcal T}}
\def \calO {{\mathcal O}}

\def \tL {{\tilde{L}}}
\def \tmu {{\tilde{\mu}}}

\def \hlambda {{\hat{\lambda}}}

\def \taut  {\text{taut}}
\def \princ {\text{princ}}
\def \tF    {{\tilde{F}}}
\def \eps   {\varepsilon}

\newcommand\id{\operatorname{id}}

\def \Symp {{\operatorname{Symp}}}
\def \Ham {{\operatorname{Ham}}}

\def\eor{\unskip\ \hglue0mm\hfill$\diamondsuit$\smallskip\goodbreak}

 \newcommand{\comment}[1]{}
 \newcommand{\ynote}[1]{}
 \newcommand{\fnote}[1]{}
 \newcommand{\mute}[2]{}
 \newcommand{\printname}[1]{}
\newcommand{\labell}[1] {\label{#1}\printname{#1}}

\begin{document}

\title{Hamiltonian group actions on exact symplectic manifolds
with proper momentum maps are standard.}

\author{Yael Karshon}
\address{Department of Mathematics, University of Toronto,
40 St.~George Street, Toronto, Ontario, M5S 2E4 Canada }
\email{karshon@math.toronto.edu}

\author{Fabian Ziltener}
\address{Department of Mathematics, Utrecht University, Budapestlaan 6, 3584CD Utrecht, The Netherlands}
\email{f.ziltener@uu.nl}

\thanks{This research was partially funded by 
the Natural Sciences and Engineering Research Council of Canada (NSERC)}

\date{\today}

\begin{abstract}
We give a complete characterization of Hamiltonian actions
of compact Lie groups on exact symplectic manifolds
with proper momentum maps.
We deduce that every Hamiltonian action of a compact Lie group
on a contractible symplectic manifold with a proper momentum map
is globally linearizable.
\end{abstract}

\maketitle

\section{The main result}
\labell{sec:result}

Let $G$ be a compact Lie group with Lie algebra $\g = T_eG$
and dual space $\g^* = T_e^* G$.
A \textbf{momentum map} for a symplectic $G$ action 
on a symplectic manifold $(M,\omega)$
is a map $\mu \colon M \to \g^*$ that intertwines the $G$ action on $M$
with the coadjoint $G$ action on $\g^*$ and that satisfies Hamilton's equation
\begin{equation} \labell{Hamiltons eq}
 d\mu^\xi = - \iota(\xi_M) \omega \qquad \text{ for all } \xi \in \g , 
\end{equation}
where $\xi_M$ is the vector field on $M$ that corresponds to $\xi$
and where $\mu^\xi := \left< \mu , \xi \right> \colon M \to \R$
is the $\xi$th component of $\mu$.
A symplectic $G$ action is called \textbf{Hamiltonian} if it admits
a momentum map.
A \textbf{Hamiltonian $\mathbf{G}$ manifold} is a triple $(M,\omega,\mu)$
where $(M,\omega)$ is a symplectic manifold with a symplectic $G$ action
and $\mu \colon M \to \g^*$ is a momentum map.
An \textbf{isomorphism} of Hamiltonian $G$ manifolds
is an equivariant symplectomorphism
that intertwines the momentum maps.
Throughout this paper, manifolds are assumed to be non-empty 
unless stated otherwise.

\begin{Remark}
This terminology is not used consistently in the literature.
We emphasize that throughout this paper we require momentum maps 
to be equivariant
and we require the two-forms of Hamiltonian $G$ manifolds 
to be non-degenerate.
\end{Remark}

\begin{Remark}
A symplectic $G$ action on a symplectic manifold $(M,\omega)$ 
is the same thing as a homomorphism from $G$ to the group
$\Symp(M,\omega)$ of symplectomorphisms of $(M,\omega)$
that is smooth in the diffeological sense:
the map $(g,x) \mapsto g \cdot x$ from $G \times M$ to $M$ is smooth
(cf.\ \cite{IK}).
The action has a momentum map if and only if 
the image of the identity component of $G$ 
is contained in the subgroup $\Ham(M,\omega)$ of Hamiltonian diffeomorphisms.
(Being contained in $\Ham(M,\omega)$ implies that there is a map 
$\mu \colon M \to \g^*$ that satisfies Hamilton's equation~\eqref{Hamiltons eq}.
The map $x \mapsto {\displaystyle \int_{a \in G} \Ad^*(a) (\mu(a^{-1}x)) da }$,
where $da$ is the Haar probability measure on $G$,
is then a momentum map.)
It is natural to require the image of all of $G$, 
not only the identity component, to be contained in $\Ham(M,\omega)$,
but we will not need this stronger requirement.
\end{Remark}

We recall a particular construction of Hamiltonian $G$ manifolds:

\begin{Definition} \labell{def:centred model}
Let $G$ be a compact Lie group with Lie algebra $\g=T_eG$
and dual space $\g^*=T_e^*G$.
A \textbf{centred Hamiltonian $G$-model} is a Hamiltonian $G$ manifold
$(Y,\omega_Y,\mu_Y)$, that is obtained by the following construction.
\begin{quotation}
Let $H$ be a closed subgroup of $G$.
Let $(V,\omega_V)$ be a symplectic vector space with a linear $H$ action
and quadratic momentum map $\mu_V \colon V \to \h^*$.
Consider the $H$ action on $T^*G$ that is induced from its action
$h \colon a \mapsto a h\inv$ on $G$, with the fibrewise homogeneous
momentum map.  The model $(Y,\omega_Y)$
is the symplectic quotient at zero of $T^*G \times V$
with respect to the anti-diagonal $H$ action.
The $G$ action and momentum map on $Y$ are induced from 
the left action $g \colon a \mapsto ga$ on $G$
and its homogeneous momentum map on $T^*G$.
\end{quotation}
\end{Definition}

We describe centred Hamiltonian $G$ models in greater detail 
in Section~\ref{sec:models}.

The purpose of this paper is to prove the following theorem.

\begin{Theorem} \labell{main}
Let a compact Lie group $G$ act on a symplectic manifold $(M,\omega)$
with momentum map $\mu \colon M \to \g^*$.  
Assume that $M/G$ is connected.
Assume that $\mu$ is proper and $\omega$ is exact.
Then $M$ is equivariantly symplectomorphic to a centred Hamiltonian $G$-model.
\end{Theorem}

We prove Theorem~\ref{main} in Section~\ref{sec:liouville}.
Here are some implications on the topology, equivariant topology,
and symplectic topology of the manifold.

\begin{Corollary} \labell{cor:properties}
Let a compact Lie group $G$ act on a symplectic manifold 
$(M,\omega)$ with momentum map $\mu \colon M \to \g^*$.  
Assume that $M/G$ is connected.
Assume that $\mu$ is proper and $\omega$ is exact.
Then the following results hold.
\begin{itemize}
\item[(i)] The Euler characteristic of $M$ is non-negative.
\item[(ii)] $M$ contains either no $G$ fixed points 
or exactly one $G$ fixed point.
\item[(iii)] The Gromov width of $(M,\omega)$ is infinite.
\end{itemize}
\end{Corollary}

Here are two situations in which we get even more precise information.

\begin{Corollary} \labell{cor:S1}
Let the circle group $S^1$ act faithfully on a connected 
exact symplectic 
manifold with a proper momentum map.  Then the manifold is equivariantly 
symplectomorphic either to the cylinder $S^1\times\R$ 
with $S^1$ acting by rotations or to $\C^n$ with $S^1$ acting linearly.
\end{Corollary}

\begin{Corollary} \labell{cor:contractible}
Let a compact Lie group $G$ act on a symplectic manifold $(M,\omega)$
with a momentum map $\mu \colon M \to \g^*$.
Suppose that $M$ is contractible and $\mu$ is proper.
Then $(M,\omega)$ is equivariantly symplectomorphic
to a symplectic vector space with a linear symplectic $G$ action.
In particular, a compact Lie group action on $\R^{2n}$
with proper momentum map is always linearizable.
\end{Corollary}

\begin{proof}[Proofs of Corollaries~\ref{cor:properties}, 
\ref{cor:S1}, and~\ref{cor:contractible}.]
These corollaries follow from Theorem~\ref{main}
and from properties of centred Hamiltonian $G$ models
that we list in Section~\ref{sec:properties}.
For Corollary~\ref{cor:properties},
see Lemmas~\ref{homotopy}, \ref{fixed set connected}, and \ref{Gromov width}.
For Corollary~\ref{cor:S1}, Lemma~\ref{circle}.
For Corollary~\ref{cor:contractible},
note that $\omega$ is exact if $M$ is contractible, and 
see Lemma~\ref{contractible model}.
\end{proof}

\begin{Example}
None of the following exact symplectic manifolds 
admits an action of a compact Lie group $G$ with a proper momentum map.
In particular, none of them admits a proper real valued function
whose Hamiltonian flow is periodic.
\begin{itemize}
\item[(a)]
The cotangent bundle of any manifold with negative Euler characteristic, 
for example, $T^*X$ where $X$ is a closed oriented connected surface
of genus $\geq 2$ or the product of such a surface
with a closed connected manifold of positive Euler characteristic.

\item[(b)]
Any open subset of $\R^{2n}$, with coordinates $x_1,y_1,\ldots,x_n,y_n$
and the standard symplectic structure $\sum_{j=1}^n dx_j \wedge dy_j$,
on which $x_1$ and $y_1$ are bounded.
\end{itemize}
This follows from Corollary~\ref{cor:properties}:
manifolds of the first form have a positive Euler characteristic,
and manifolds of the second form have a finite Gromov width \cite{Gr}.
\end{Example}

Properness of the momentum map is crucial.   
Without this assumption, actions of the trivial group 
provide counterexamples to the conclusions of our various results.
In Section~\ref{sec:nonproper} we give more interesting examples, 
of exotic Hamiltonian actions on symplectic vector spaces.
For example, for any compact connected non-abelian Lie group $G$,
there exists a symplectic vector space $\R^{2n}$ 
with a Hamiltonian $G$ action that is not isomorphic to a linear action.
Thus, the conclusion of Corollary~\ref{cor:contractible} 
fails to hold in this situation.  See Corollary~\ref{not linearizable}.

\medskip

The strategy of the proof of Theorem \ref{main} is as follows. 
After shifting the momentum map by a constant,
we may assume that the momentum map is obtained from an invariant primitive
of $\omega$ as described in Section~\ref{sec:exact}.
Using Reyer Sjamaar's de Rham theory for singular symplectic quotients 
\cite{sjamaar}, 
we show that $\mu^{-1}(0)$ is a finite union $\bigsqcup G \cdot x_j$
of $G$-orbits. 
By the local normal form theorem for isotropic orbits 
there exists a $G$-equivariant symplectomorphism $F$ between
some neighbourhood $U$ of $\mu^{-1}(0)$ and an open subset $V$ 
in a finite union $\bigsqcup Y_j$
of centred Hamiltonian $G$-models $Y_j$. The negative Liouville flows 
on $M$ and on $Y_j$ shrink any compact set into $U$ and $V$, respectively. 
Combining them with $F$, we obtain a $G$-equivariant symplectomorphism 
between $M$ and $\bigsqcup Y_j$.
If $M/G$ is non-empty and connected, there is exactly one $Y_j$.

\medskip

In the literature, there are many classification results
for Hamiltonian group actions whose complexity -- 
half the dimension of a generic non-empty reduced space -- is low.
See
\cite{delzant,KL} (symplectic toric manifolds,
even without compactness or properness),
\cite{moser} (compact symplectic 2-manifolds, no action),
\cite{k:periodic} 
(Hamiltonian circle actions on compact symplectic four-manifolds),
\cite{KT:globex} (``tall'' complexity one Hamiltonian torus actions, 
proper momentum maps),
\cite{knop,losev} (complexity zero non-abelian group actions on compact 
symplectic manifolds),
\cite{chiang} (complexity one non-abelian group actions 
on compact symplectic six-manifolds).
Additional references are listed in the introduction to~\cite{KT:globex}.

What is special about our Theorem~\ref{main} is that it gives a situation 
in which we can characterize Hamiltonian $G$ manifolds
of \emph{arbitrary} complexity.
This work follows two earlier results that apply to actions
of arbitrary complexity:
\begin{itemize}
\item[(i)]
Delzant \cite[section~1]{delzant} proved that a Hamiltonian circle action
on a compact symplectic $2n$-manifold whose fixed point set has exactly two
connected components, of which one is an isolated fixed point,
is equivariantly symplectomorphic to a standard circle action on $\CP^n$.

\item[(ii)]
Let a torus $T$ act on a compact symplectic manifold $(M,\omega)$
with a momentum map $\mu \colon M \to \t^*$.
Let $\calT$ be a convex open subset of $\t^*$
that contains $\mu(M)$ and such that $\mu$ is proper as a map to $\calT$.
Let $x \in M$ be a point;
suppose that $\mu^{-1}(\mu(x)) = \{ x \}$
and that $\mu(x)$ is contained in the momentum image
of every component of the fixed point set $M^K$
for every subgroup $K$ of $T$.
Then $M$ is equivariantly symplectomorphic to 
the open subset 
$\{ z \in \C^n \ | \ \mu(x) + \pi \sum_j |z_j|^2 \eta_j \in \calT \}$
of the standard $\C^n$,
where $\eta_1,\ldots,\eta_n$ are the isotropy weights at $p$,
and where our normalization convention is such that 
we identify the circle with $\R/\Z$ (rather than $\R/2\pi\Z$).
This is proved in \cite[section~2]{KT:globex}.
\end{itemize}

We end this section with a discussion of the adjective ``centred''.

\begin{Remark}
Let $(M,\omega,\mu)$ be a Hamiltonian $G$ manifold.
Suppose that $\mu$ is proper as a map to some convex open subset $\calT$
of $\g^*$.  
Let $\alpha \in \g^*$ be a point that is fixed under the coadjoint action.
We say that the Hamiltonian $G$ manifold is \textbf{centred} about $\alpha$
if, for every subgroup $K$ of $G$,
the point $\alpha$ is contained in the momentum map image 
of every component of the fixed point set $M^K$.
\begin{itemize}
\item[(i)]
Every centred Hamiltonian $G$ model
whose momentum map is proper (to $\calT := \g^*$)
is centred about $\alpha = 0$ according to this definition.
\item[(ii)]
This definition of ``centred'' is consistent with the notion of ``centred''
that was introduced in \cite[Def.~1.4]{KT:centered} 
and used in \cite{KT:gromov}.
The only difference is that in \cite{KT:centered,KT:gromov}
we restricted to the special case that $G$ is a torus and $M$ is connected
and we fixed a choice of $\calT$.
(The definition in \cite[Def.~1.4]{KT:centered}
is phrased slightly differently from that of \cite[Def.~2.2]{KT:gromov},
but it is equivalent to it.  See \cite[Rk.~2.7]{KT:gromov}.)
\end{itemize}
\eor
\end{Remark}

\section{Exact Hamiltonian $G$ manifolds}
\labell{sec:exact}

A symplectic manifold $(M,\omega)$ is called \textbf{exact}
if the symplectic form $\omega$ is exact, that is,
if it has a primitive:
a one-form $\lambda$ such that $\omega = d\lambda$.
If an exact form $\omega$ is preserved under the action of a 
compact Lie group $G$ then, by averaging,
we can choose its primitive $\lambda$ to be $G$ invariant.
In this situation, consider the map
$$ \mu \colon M \to \g^* $$
such that, for every $\xi \in \g^*$, the component
$$ \mu^\xi := \left< \mu , \xi \right> \colon M \to \R $$
is given by plugging the corresponding vector field $\xi_M$
into the one-form $\lambda$:
\begin{equation} \labell{exact mm}
 \mu^\xi = \iota(\xi_M) \lambda.
\end{equation}
This map
is a momentum map for the $G$-action on $(M,\omega)$.
(Indeed, by Cartan's formula 
$\calL_{\xi_M} \lambda 
  = \left( d \iota (\xi_M) + \iota(\xi_M) d \right) \lambda.$
The right hand side is $d \mu^\xi + \iota(\xi_M) \omega$,
and the left hand side is zero.)
A Hamiltonian $G$ manifold $(M,\omega,\mu)$ is \textbf{exact}
if $\mu$ is obtained in this way from an invariant primitive of $\omega$.

The following lemma is used in the proof of Lemma \ref{structure descends}.
A variant of it for singular reduced spaces appears in 
Lemma~\ref{omega exact}.

\begin{Lemma}[Reduction of exact symplectic manifolds] 
\labell{smooth reduction} 
Let $G$ be a compact Lie group and let $(M,\omega,\mu)$ be 
an exact Hamiltonian $G$ manifold.
Suppose that the $G$ action on the level set $Z := \mu^{-1}(0)$ is free.
Consider the inclusion-quotient diagram 
$$ \xymatrix{
 Z \ar@{^{(}->}[r]^{\text{i}} \ar[d]_{\pi} & M. \\ M_\red &
    } $$
Then  
\begin{itemize} 
\item[(i)]
there exist unique smooth manifold structures on $Z$ and $M_\red$ 
such that $i$ is an immersion and $\pi$ is a submersion; 
\item[(ii)]
there exist a unique one-form $\lambda_\red$
and a unique two-form $\omega_\red$ on $M_\red$ such that 
$$ \pi^* \lambda_\red = i^* \lambda \qquad \text{ and } \qquad
   \pi^* \omega_\red = i^* \omega ;$$
\item[(iii)]
the two-form $\omega_\red$ on $M_\red$ is symplectic, and
$\omega_\red = d\lambda_\red$.  
\end{itemize} 
\end{Lemma}

\begin{proof}
The only difference from the standard Marsden-Weinstein reduction \cite{MW} 
is that we also obtain a reduction of the one-form $\lambda$ and not only 
of the two-form $\omega$.   Indeed, the one-form $i^*\lambda$ on $Z$ 
is $G$ invariant because $\lambda$ is $G$-invariant,
and it is $G$-basic (that is, descends to $M_\red$)
because, additionally, $\iota_{\xi_M} i^* \lambda = \mu^\xi \circ i = 0$
for all $\xi \in \g$.
\end{proof}

\section{Centred Hamiltonian $G$ models}
\labell{sec:models}

In this section we describe centred Hamiltonian $G$ models
$(Y,\omega_Y,\mu_Y)$
with more details than in Definition~\ref{def:centred model},
and on every such a model we specify a particular $G$-invariant
one-form $\lambda_Y$
such that $d\lambda_Y = \omega_Y$
and whose corresponding momentum map is $\mu_Y$.

\medskip

The construction involves three ingredients:
a compact Lie group $G$, a closed subgroup $H \subset G$,
and a linear symplectic representation of $H$
on a symplectic vector space $(V,\omega_V)$
with a quadratic momentum map $\mu_V \colon V \to \h^*$.

\medskip

We begin with some notation.
We denote an element of the cotangent bundle $T^*G$
by either $\beta$ or $(a,\beta)$, interchangeably,
where $\beta \in T_a^*G$.
For $g \in G$, we denote by $(a,\beta) \mapsto (ga,g\beta)$
the lifting to $T^*G$ of the diffeomorphism $a \mapsto ga$ of $G$.
Similarly, we denote by $(a,\beta) \mapsto (ah,\beta h)$
the lifting to $T^*G$ of the diffeomorphism $a \mapsto ah$ of $G$.
(That is, if $L_{g'} \colon G \to G$ and $R_{h'} \colon G \to G$ 
denote, respectively, the left and right translation maps
$a \mapsto g'a$ and $a \mapsto ah'$,
then $g\beta$ and $\beta h$ are the images of $\beta$
under the linear maps $L_{g\inv}^* \colon T_a^* G \to T_{ga}^* G$
and $R_{h\inv}^* \colon T_a^* G \to T_{ah}^* G$.)
These transformations commute, so expressions such as $g \beta h$
are well defined.
With this notation,
the left invariant trivialization of the cotangent bundle
carries $(a,\varphi) \in G \times \g^*$ to $(a,a\varphi) \in T^*G$,
and the coadjoint action of $a \in G$ on $\g^* = T_e^*G$
is $\Ad^*(a) \colon \varphi \mapsto a \varphi a\inv$.

Let $\lambda_{\taut}$ be the tautological one-form on $T^*G$,
so that $d\lambda_{\taut} = \omega_{T^*G}$
is the canonical symplectic form.
The $G$ action on $T^*G$
that is induced from the action $g \mapsto L_g$ on~$G$
has the momentum map
$$ \mu_L \colon T^*G \to \g^* \quad ; \quad
   (a,a\varphi) \mapsto \Ad^*(a)(\varphi).$$
The $H$ action on $T^*G$
that is induced from the action $h \mapsto R_{h\inv}$ on~$G$
has the momentum map
$$ \mu_R \colon T^*G \to \h^* \quad ; \quad
   (a,a\varphi) \mapsto - \varphi|_\h ,$$
where $\varphi \mapsto \varphi|_\h$ is the natural map $\g^* \to \h^*$.
Moreover, the momentum maps that are obtained from $\lambda_{\taut}$
vanish along the zero section, 
so they coincide with the momentum maps $\mu_L$ and $\mu_R$.

Before proceeding, we recall the definition of the Euler vector field.

\begin{Definition} \labell{def:euler}
The \textbf{Euler vector field} on a vector space $W$
is the vector field that, under the natural trivialization of $TW$,
is given by the identity map. 
Equivalently, it is the velocity vector field of the flow 
$t \mapsto e^t v$ on $W$.
\end{Definition}

Note that, if $X_V$ is the Euler vector field on a vector space $V$
and $\alpha$ is a differential form on $V$ with constant coefficients,
then $\LL_{X_V} \alpha = k \alpha$ where $k$ is the degree of $\alpha$.

We now turn to the symplectic vector space $(V,\omega_V)$ 
with the $H$ action and 
with the quadratic momentum map $\mu_V \colon V \to \h^*$.
Let $\lambda_V = \frac{1}{2} \iota(X_V) \omega_V$,
where $X_V$ is the Euler vector field on $V$. 
Then $\lambda_V$ is an $H$-invariant one-form on $V$, 
and $d\lambda_V = \omega_V$. 
(The latter follows from Cartan's formula, closedness of $\om_V$, 
and the identity $\LL_{X_V}\om_V=2\om_V$.) 
Moreover, the momentum map that is obtained from $\lambda_V$
vanishes at the origin, 
so it coincides with the quadratic momentum map $\mu_V$.

We now consider the product $T^*G \times V$, 
whose elements we write as $(a,a\varphi,v)$
with $a \in G$, $\varphi \in \g^*$, and $v \in V$.
We have the left $G$ action
$ \tL_g \colon (a,a\varphi,v) \mapsto (ga, ga\varphi , v)$,
with the momentum map
\begin{equation} \labell{tmuL}
 \tmu_L \ \colon \ T^*G \times V \ \to \ \g^* \quad ; \quad
 (a,a\varphi,v) \mapsto \Ad^*(a)(\varphi) ,
\end{equation}
and the anti-diagonal $H$ action
$ D_h \colon (a,a\varphi,v) \mapsto (ah\inv, a \varphi h\inv , h \cdot v)$,
with the momentum map
\begin{equation} \labell{tmuD}
 \tmu_D \ \colon \ T^*G \times V \ \to \ \h^* \quad ; \quad
 (a,a\varphi,v) \mapsto - \varphi|_\h + \mu_V(v) .
\end{equation}
The left $G$ action and anti-diagonal $H$ actions commute.
The $H$ momentum map $\tmu_D$ is invariant with respect to the $G$ action,  
and the $G$ momentum map $\tmu_L$ 
is invariant with respect to the anti-diagonal $H$ action.

In Definition~\ref{def:centred model},
we defined the centred Hamiltonian $G$ model $(Y,\omega_Y)$
that is constructed from this data
to be the symplectic reduction of $T^*G \times V$
with respect to the $H$ action.
Explicitly,
$$ Y \ = \ (T^*G \times V) /\!/ H \ := \ Z / H$$
where 
$$ Z := \tmu_D \inv (0)
\left( = 
\left\{ (a,a\varphi,v) \ \left| \ 
 \varphi(\xi)=\mu_V^\xi(v) \text{ for all } \xi \in \h \right. \right\} 
 \right). $$

\begin{Lemma} \labell{structure descends} \ 
\begin{enumerate}
\item
There exist unique manifold structures on $Z$ and on $Y$
such that the inclusion map $i \colon Z \to T^*G \times V$ is an immersion
and the quotient map $\pi \colon Z \to Y$ is a submersion.

\item
There exist a unique one-form $\lambda_Y$ and a unique two-form $\omega_Y$
on $Y$ such that $\pi^* \lambda_Y = i^* (\lambda_{\taut} \oplus \lambda_V)$
and $\pi^* \omega_Y = i^* (\omega_{T^*G} \oplus \omega_V)$.

\item
The two-form $\omega_Y$ is symplectic, and $d\lambda_Y = \omega_Y$.
\end{enumerate}

Moreover, 

\begin{enumerate}[resume]
\item
The left $G$ action on $T^*G \times V$
descends to a $G$ action on the model $Y$.

\item\labell{mu Y}
The $G$ momentum map 
$\tmu_L \colon T^*G \times V \to \g^*$
descends to a $G$ momentum map 
$$ \mu_Y \colon Y \to \g^* . $$

\item
The $G$ momentum map $\mu_Y$ coincides with the momentum map 
that is obtained from the one-form $\lambda_Y$.
\end{enumerate}

\end{Lemma}

\begin{proof}
For items (1)--(3),
note that $H$ acts freely on the zero level set $Z$,
and apply Lemma~\ref{smooth reduction}.
Item (4) follows from the fact that, on $T^*G \times V$,
the left $G$ action preserves the $H$ momentum map $\tmu_D$ 
and commutes with the anti-diagonal $H$ action.
Item (5) follows from the fact that $\tmu_L$ is $H$ invariant.
Item (6) follows from the fact that, on $T^*G \times V$,
the $G$ momentum map $\tmu_L$ 
coincides with the momentum map that is obtained from 
the one-form $\lambda_\taut \oplus \lambda_V$.
\end{proof}

We now record a property of centred models
that we use in Lemma \ref{le:local} below.

\begin{Definition} \labell{central orbit}
The \textbf{central orbit} in a centred Hamiltonian $G$-model
$Y = (T^*G \times V) /\!/ H$ is the image in $Y$ of 
(the zero section of $T^*G$)$\times$(the origin of $V$).
\end{Definition}

\begin{Lemma} \labell{muY is proper}
Let $(Y,\omega_Y,\mu_Y)$ be a centred Hamiltonian $G$ model.
Suppose that there exists a neighbourhood $U'$ of the central orbit in $Y$
such that $U' \cap \mu_Y\inv(0)$ consists of only the central orbit.
Then $\mu_Y\inv(0)$ consists of only the central orbit,
and the map $\mu_Y \colon Y \to \g^*$ is proper.
\end{Lemma}

\begin{proof}
In the notation of~\eqref{tmuL} and~\eqref{tmuD},
$\mu_Y\inv(0)$ is the quotient of
$\tmu_L\inv(0) \cap \tmu_D\inv(0)$ by the anti-diagonal $H$ action.
Because $\tmu_L\inv(0) \cap \tmu_D\inv(0) = (\text{the zero section
 in } T^*G) \times \mu_V\inv(0) $
and $\mu_V$ is quadratic, we deduce that
$\mu_Y\inv(0)/G \left( \cong \mu_V\inv(0)/H \right) $ is connected.

By assumption, the central orbit is an isolated point in $\mu_Y\inv(0)$/G.  
This implies that $\mu_Y\inv(0)$ consists of only the central orbit.

In particular, $\mu_V\inv(0) = \{ 0 \}$.
So $m := \min \big\{  \| \mu_V(v) \| \ | \ \| v \| = 1  \big\}$
(with respect to any norm on $V$) is positive.
By homogeneity, $\| \mu_V(v) \| \geq m \| v\|^2$ for all $v$.
This implies that $\mu_V \colon V \to \h^*$ is proper.
From this and the compactness of $G$ we further deduce
that $(\tmu_L,\tmu_D) \colon T^*G \times V \to \g^* \times \h^*$,
and hence $\mu_Y \colon Y \to \g^*$, is proper.
\end{proof}

\medskip

We end this section with an alternative description of the model $Y$
that we use in Section~\ref{sec:properties}.

Fix an $\Ad$-invariant inner product on $\g$, 
and use it to embed $\h^*$ in $\g^*$ as the orthogonal complement
of $\h^0$, the annihilator of $\h$ in $\g^*$.
Consider the action of $h \in H$ on $G \times (\h^0 \oplus V)$
by right multiplication by $h^{-1}$ on the $G$ factor,
by the coadjoint action on the $\h^0$ factor,
and by the given representation on the $V$ factor.
Our alternative model is the quotient,
$$G \times_H \left( \h^0 \times V \right),$$
with $G$ acting by left multiplication on the $G$ factor.
We have a pull-back diagram:
$$\xymatrix{
 G \times (\h^0 \times V) \ar[r] \ar[d] & \tmu_D\inv(0) \ar[d] 
                            & \subset T^*G \times V \\
 G \times_H ( \h^0 \times V ) \ar[r] & (T^*G \times V)/\!/H  & = Y &
}$$
The bottom right term is the model $Y$.
The top arrow is the map 
\begin{equation}\labell{a nu v}
  (a,\nu,v) \mapsto (a, \Ad^*(a)(\nu+\mu_V(v)),v).
\end{equation}
The bottom arrow is a $G$ equivariant diffeomorphism;
pulling back $\omega_Y$ by it, we get a $G$ invariant symplectic form
on $G \times_H (\h^0\oplus V)$ 
whose pullback to $G \times (\h^0 \oplus V)$
coincides with the pullback of $\omega_{T^*G} \oplus \omega_V$
under the top arrow.

\begin{Lemma}\ \labell{alternative model}
\begin{enumerate}
\item\labell{omega V} The pullback of this symplectic form 
on $G \times_H (\h^0 \times V)$ 
under the inclusion map $v \mapsto [1,0,v]$ of $V$
is $\omega_V$. 
\item\labell{isotropic} The zero section $G \times_H \{ 0 \} $ 
of $G \times_H (\h^0 \times V)$ is isotropic.
\item\labell{momentum map} The momentum map for the $G$-action 
on $G \times_H (\h^0 \times V)$ with respect to the symplectic form 
described above is
\[[a,\nu,v] \mapsto\Ad^*(a)(\nu + \mu_V(v)).\]
\end{enumerate}
\end{Lemma}

\begin{proof}To prove item (\ref{omega V}), observe that this pullback form agrees with the pullback of $\om_{T^*G}\oplus\om_V$ under the map $V\ni v\mapsto(1,\mu_V(v),v)\in T^*G\x V$. This pullback equals $\om_V$, since the fiber $T^*_1G$ is Lagrangian in $T^*G$. This proves (\ref{omega V}).

Item (\ref{isotropic}) follows from the definition of the symplectic form. Item (\ref{momentum map}) follows from the fact that this momentum map is the composition of $\mu_Y$ (defined by (\ref{tmuL}) and Lemma \ref{structure descends}(\ref{mu Y})) with the map induced by (\ref{a nu v}).
\end{proof}

\bigskip

{

The notion of a centred Hamiltonian $G$ model comes from the local normal
form theorem.  
(A special case of) the local normal form theorem can then be stated 
as follows.

\begin{Theorem} \labell{lnf isotropic} 
Let a compact Lie group $G$ act on a symplectic manifold $(M,\omega)$
with a ($G$-equivariant) momentum map $\mu \colon M \to \g^*$,
and let $G \cdot x$ be an orbit that is contained in the 
zero level set $\mu\inv(0)$.
Then there exist a centred Hamiltonian $G$-model $Y$ and
a $G$-equivariant symplectomorphism from an invariant open neighbourhood
of $G \cdot x$ in $M$
to an invariant open subset in $Y$
that takes $G \cdot x$ to the central orbit in $Y$.
\end{Theorem}

The local normal form theorem is the main ingredient 
in the work of Sjamaar-Lerman \cite{SL}
and in Sjamaar's de Rham theory for symplectic quotients
\cite{sjamaar} (see Section~\ref{sec:sjamaar}).
The theorem in its more general form applies to neighbourhoods
of orbits that are not necessarily in the zero level set $\mu\inv(0)$
and, unless the orbit is isotropic, it involves Hamiltonian $G$ models
that are more general than the centred Hamiltonian $G$ models
of Definition \ref{def:centred model}.
The theorem, due to Guillemin-Sternberg and Marle,
can be found in \cite{GS1984,marle1985}.
We use the local normal form theorem in Section~\ref{sec:sjamaar} 
and in the proof of Lemma~\ref{le:local}.

%

\section{The zero level set of the momentum map and Sjamaar's de Rham theory 
for symplectic quotients}
\labell{sec:sjamaar}

The purpose of this section is to prove Proposition~\ref{sjamaar},
which is later used in the proof of Lemma \ref{le:local}.

We will use a de Rham theory
for singular symplectic quotients that was introduced by R.~Sjamaar 
in~\cite{sjamaar}. To explain it, we first recall some facts 
from the paper~\cite{SL} of Sjamaar and Lerman;
specifically, see \cite [Theorems 2.1 and 5.9]{SL}.

Let $G$ be a compact Lie group,
and let $(M,\omega,\mu)$ be a Hamiltonian $G$ manifold.

The reduced space $\mu\inv(0)/G$ is locally connected; 
this is a consequence of the local normal form theorem 
(Theorem \ref{lnf isotropic}).
So the connected components of the reduced space
are open and closed in the reduced space,
and their preimages in $M$ are closed in $M$.

Let $x \in \mu\inv(0)/G$.
There exist 
a conjugacy class $C$ of closed subgroups of $G$,
and a neighbourhood $U$ of $x$ in $\mu\inv(0)/G$,
such that the set of orbits in $U$ with stabilizers in $C$
is open and dense in $U$.
(We say that a $G$ orbit has stabilizers in $C$ 
if the stabilizers of the points in the orbit are in $C$.
If this is true for one point in the orbit, then it is true 
for all the points in the orbit.)

Fix a connected component $X$ of $\mu\inv(0)/G$.

It follows that there exists a unique conjugacy class $C_X$ 
of closed subgroups of $G$ such that the ``principal stratum''
$$ X_\princ := \{ x \in X \ | \ x \text{ has stabilizers in $C_X$ } \}$$
is open and dense in $X$.
Moreover, $X_\princ$ is a manifold, in following sense.
Let $Z_\princ$ be the preimage of $X_\princ$ in $\mu\inv(0)$.
Then there exist unique manifold structures on $X_\princ$ and on $Z_\princ$
such that the inclusion map $Z_\princ \to M$ is an immersion
and such that the projection map $Z_\princ \to X_\princ$ is a submersion.

The \textbf{dimension} of $X$ is the dimension of $X_\princ$.

Following Sjamaar \cite{sjamaar}, 
we define a \textbf{differential form on $X$} to be 
a differential form on $X_\princ$ whose pullback to $Z_\princ$ 
extends to a differential form on $M$. 
(Because the preimage $Z$ of $X$ in $\mu\inv(0)$ is closed in $M$,
a differential form on $Z_\princ$ extends to a differential form to $M$
if and only if it every point on $Z$ has a neighbourhood
on which this differential form extends.)

The space of differential forms on $X$ is a differential graded algebra,
with the usual operations of exterior derivative and wedge products.
We denote it $\Omega(X)$.

The symplectic form $\omega$ on $M$ descends to a two-form $\omega_X$ on $X$.
When $Z$ is a manifold and the $G$ action on $Z$ is free,
this is exactly the Marsden-Weinstein reduction.

\begin{Lemma} \labell{omega not exact}
Let $G$ be a compact Lie group, 
let $(M,\omega,\mu)$ be a Hamiltonian $G$ manifold,
and let $X$ be a connected component of $\mu\inv(0)/G$.  
If $X$ is compact and of dimension $\geq 2$,
then the two-form $\omega_X$ is not exact 
in the differential complex $\Omega(X)$.
\end{Lemma}

\begin{proof}
Let $k = \frac{1}{2} \dim X$.  Assume that $k\geq1$. 
By~\cite[Corollary~7.6]{sjamaar}, 
the class of $\omega_X^k$ in the cohomology of the differential complex 
$\Omega(X)$ is nonzero. 
So $\omega_\red^k$, and hence $\omega_\red$, is not exact.
\end{proof}

\begin{Lemma} \labell{omega exact}
Let $G$ be a compact Lie group, let $(M,\omega,\mu)$ 
be an exact Hamiltonian $G$ manifold,
and let $X$ be a connected component of $\mu\inv(0)/G$.
Then the reduced form $\omega_X$ is exact in $\Omega(X)$.
\end{Lemma}

\begin{proof}
Let $\lambda$ be a $G$ invariant one-form on $M$ such that $\omega = d\lambda$
and such that $\mu^\xi = \lambda(\xi_M)$ for all $\xi \in \g$.
Let $Z_\princ$ be the preimage of $X_\princ$ in $\mu\inv(0)$.
Because $\mu=0$ on $Z_\princ$,
the pullback of the one-form $\lambda$ to $Z_\princ$ is horizontal
and therefore $G$-basic. 
Hence this pullback descends to a one-form $\lambda_X$ on $X$.
Because $\omega=d\lambda$ on $M$, we have $\omega_X = d \lambda_X$ on $X$.
\end{proof}

\begin{Proposition} \labell{sjamaar}
Let $G$ be a compact Lie group,
and let $(M,\omega,\mu)$ be an exact Hamiltonian $G$ manifold.
Let $X$ be a connected component of the reduced space $\mu\inv(0)/G$.
Suppose that $X$ is compact.  
Then $X$ consists of one point, which is isolated in $\mu\inv(0)/G$.
\end{Proposition}

\begin{proof}
By Lemma~\ref{omega not exact}, if $X$ is compact and $\dim X > 0$
then $\omega_X$ is not exact in $\Omega(X)$.
By Lemma~\ref{omega exact}, $\omega_\red$ is exact in $\Omega(X)$.
Thus, $\dim X = 0$.
Because $X$ is connected, $X$ consists of a single point.
Because $\mu\inv(0)/G$ is locally connected,
this point is isolated in $\mu\inv(0)/G$.
\end{proof}

\begin{Remark}
It follows that, if $M/G$ is connected and $\mu$ is proper,
then $\mu\inv(0)$ consists of at most one $G$ orbit.
This is because these additional assumptions imply that $\mu\inv(0)/G$ 
is connected.
(This follows from Kirwan's Morse-type theory
for the norm-square of the momentum map,
or it can be deduced from connectedness and convexity results
for abelian groups as in \cite{LMTW}. 
In \cite{LMTW} it is assumed that $G$ and $M$ are connected;
we get the general case by applying this special case 
to the action of the identity component of $G$
on a connected component of $M$.)
\eor
\end{Remark}

\section{Retracting $M$}
\labell{sec:retracting}

In our main theorem, Theorem \ref{main}, 
we assumed that the momentum map $\mu \colon M \to \g^*$ is proper.
When $(M,\omega,\mu)$ is exact, it is enough to assume that
$\mu$ is proper as a map to an open subset $\calT$ of $\g^*$
that is starshaped about the origin.  
This weaker property comes up,
for example, when we start with a Hamiltonian $G$ action
on a compact symplectic manifold
and obtain $M$ by restricting to the momentum map preimage of $\calT$.
From now on we will work with this weaker assumption,
but the reader is welcome to restrict attention everywhere 
to the special case $\calT = \g^*$.

We recall some facts about vector fields and flows.
Good references are the textbooks by John Lee \cite{johnlee}
and by Br\"ocker and J\"anich \cite{BJ}.

Recall that a \textbf{flow} on a manifold or manifold with boundary $Y$
is a smooth map $(t,y) \mapsto \rho_t(y)$ to $Y$,
defined on a \textbf{flow domain} -- an open subset $D$
of $\R \times Y$ such that for every $y \in Y$
the set $\{ t \in \R \ | \ (t,y) \in D \}$ 
is a interval that contains the origin -- 
such that $\rho_{t+s}(y) = \rho_t(\rho_s(y))$,
in the sense that if the right hand side is well defined
then so is the left hand side and they are equal.  
We say that a flow is defined for all $t \geq 0$
if its flow domain contains $\R_{\geq 0} \times Y$.

By the fundamental theorem of ordinary differential equations,
if $X$ is a vector field on $Y$, then there exists a unique flow on $Y$,
called the \textbf{flow of $X$}, that has the following properties.
For any $x_0 \in X$, interval $I$ of the form $[0,b)$ or $(a,0]$,
and smooth curve $x \colon I \to Y$,
the curve satisfies $x(0) = x_0$ and $\dot x = X \circ x$
if and only if $I \subseteq \{ t \in \R \ | \ (t,x_0) \in D \}$
and $x(t) = \rho_t(x_0)$ for all $t \in I$.

Given a map $f \colon Y_1 \to Y_2$, we say that a flow $\rho^{Y_1}_t$
on $Y_1$ \textbf{lifts} a flow $\rho^{Y_2}_t$ on $Y_2$
if the flow domain of $\rho^{Y_1}_t$ is the preimage
under $(t,y) \mapsto (t,f(y))$ of the flow domain of $\rho^{Y_2}_t$
and if $f(\rho^{Y_1}_t(y)) = \rho^{Y_2}_t(f(y))$
whenever the left and right hand sides of this equation are defined. %

We recall that if an integral curve for a vector field 
does not exist for all times then it leaves every compact set:

\begin{lemma}\labell{le:extend} 
Let $M$ be a smooth manifold without boundary, $X$ a vector field on $M$, 
\ $x_0\in M$, $t > 0$, and $x \colon [0,t)\to M$ a smooth curve such that
\begin{equation}\labell{eq:x 0 dot x}
   x(0)=x_0,\quad\dot x=X\circ x.
\end{equation}
If the closure of $x([0,t))$ is compact, then $x$ extends 
to a smooth solution of (\ref{eq:x 0 dot x}) that is defined 
on $[0,t')$ for some number $t'>t$.
\end{lemma}

\begin{proof}
This follows for example from the argument on p.~84 in \cite{BJ}
or in the ``Escape Lemma'' \cite[Lemma~17.10]{johnlee}.
\end{proof}

The next lemma will be used in the proofs of Lemmata \ref{retraction} 
and \ref{liouville} below.

\begin{Lemma} \labell{flow lifts}
Let $X_1$ be a vector field on $Y_1$, let $X_2$ be a vector field on $Y_2$,
and let $f \colon Y_1 \to Y_2$ be a smooth map whose differential intertwines
$X_1$ with $X_2$. Then the image of the flow domain of $X_1$ under the map $\id_\R\times f$ is contained in the flow domain of $X_2$. 

If $f$ is proper then the flow of $X_1$ lifts the flow of $X_2$. 
In particular, if the flow of $X_2$ is defined for all $t \geq 0$ 
then so is the flow of $X_1$.
\end{Lemma}

\begin{proof}
Let $\rho^{Y_i} \colon D_i \to Y_i$ be the flow of $X_i$.
Fix any $y_1^0 \in Y_1$ and $y_2^0 \in Y_2$ such that $f(y_1^0) = y_2^0$.
Let $(a_i,b_i) = \{ t \in \R \ | \ (t,y_i^0) \in D_i \}$,
and define $y_i \colon (a_i,b_i) \to Y_i$
by $y_i(t) = \rho^{Y_i}_t(y_i^0)$.
We claim that $(a_1,b_1) \subseteq (a_2,b_2)$
and that for all $t$ in this interval we have $f(y_1(t)) = y_2(t)$.

Because $f(y_1^0) = y_2^0$
and the differential of $f$ intertwines $X_1$ with $X_2$,
and because the curve $y_1$ satisfies 
\begin{equation} \labell{properties of y1}
y_1(0) = y_1^0 \ \text{ and } \ \dot y_1 = X_1 \circ y_1,
\end{equation}
the curve $y_2' := f \circ y_1 \colon (a_1,b_1) \to Y_2$
satisfies $y_2'(0) = y_2^0$ and $\dot y_2' = X_2 \circ y_2$.
By the properties of the flow of $X_2$,
it follows that $(a_1,b_1) \subseteq (a_2,b_2)$
and $y_2'(t) = y_2(t)$ for all $t \in (a_1,b_1)$.

Assume now that $f$ is proper.
It remains to show that $(a_1,b_1) = (a_2,b_2)$.
Seeking a contradiction, assume that $b_1 < b_2$; 
the case $a_1 > a_2$ is similar.
Then $[0,b_1] \subset [0,b_2)$.
Since $f$ is proper, the set $f\inv(y_2([0,b_1]))$ is compact.
It contains $y_1([0,b_1))$.
It follows from Lemma~\ref{le:extend}
that $y_1$ extends to a solution of~\eqref{properties of y1}
that is defined on $[0,b_1')$ for some $b_1' > b_1$.
This contradicts the definition of $b_1$ and properties of the flow of $X_2$.
\end{proof}

\begin{Definition} \labell{def:liouville}
Let $(M,\omega)$ be an exact symplectic manifold
and $\lambda$ a one-form such that $d\lambda = \omega$.
The corresponding \textbf{Liouville vector field}
is the vector field $X_M$ on $M$ that satisfies $\iota(X_M) \omega = \lambda$.
\end{Definition}

The following lemma will be used in the proofs of 
Lemmata \ref{meets B and extend oneform} and \ref{liouville}.

\begin{Lemma} \labell{lem:liouville}
Let $G$ be a compact Lie group,
and let $(M,\omega,\mu)$ be an exact Hamiltonian $G$ manifold.
Let $\lambda$ be a $G$ invariant one-form on $M$
such that $d\lambda = \omega$ and $\mu^\xi = \iota(\xi_M)\lambda$
for all $\xi \in \g$.
Let $X_M$ be the corresponding Liouville vector field.  Then
\begin{itemize}
\item $X_M$ is $G$ invariant.
\item $\LL_{X_M} \lambda = \lambda$. 
\item The momentum map $\mu$ intertwines $X_M$ with the Euler vector field
on $\g^*$ (see Definition \ref{def:euler}).
\end{itemize}
\end{Lemma}

\begin{proof}\ 
\begin{itemize}
\item
The $G$ invariance of $X_M$ follows from the $G$ invariance of $\lambda$
and of $\omega$.
\item
We have
$$ \LL_{X_M} \lambda = \iota(X_M) d \lambda + d \iota(X_M) \lambda
 = \iota(X_M) \omega + d \iota(X_M) \iota(X_M) \omega
 = \iota(X_M) \omega + 0 = \lambda. $$

\item
For each $\xi \in \g$ we have
$$ \LL_{X_M} \mu^\xi = \iota(X_M) d \mu^\xi
 = - \iota(X_M) \iota(\xi_M) \omega
 = \iota(\xi_M) \iota(X_M) \omega
 = \iota(\xi_M) \lambda
 = \mu^\xi .$$
This implies that $\mu$ intertwines $X_M$ 
with the Euler vector field on $\g^*$.
\end{itemize}
\end{proof}

\smallskip

Fix an $\Ad^*$-invariant inner product on $\g^*$. 
The next lemma will be used in the proof of 
Lemma~\ref{meets B and extend oneform}.

\begin{Lemma} \labell{retraction}
Let $M$ be a $G$ manifold and $\mu \colon M \to \g^*$ a map
that intertwines the $G$ action on $M$ with the coadjoint $G$ action on $\g^*$.
Let $\calT$ be an open subset of $\g^*$
that is $\Ad^*$-invariant and is starshaped about the origin. 
Suppose that the image of $\mu$ is contained in $\calT$
and that $\mu$ is proper as a map to $\calT$.

Let $V$ be an open ball in $\g^*$ centred at the origin 
and contained in $\calT$,
and let $B$ be a closed ball in $\g^*$ centred at the origin
and contained in $V$.
Suppose that there exists a vector field on $M \ssminus \mu^{-1}(0)$
that lifts the Euler vector field on $\g^*$.

Then there exists a $G$ equivariant map $\varphi \colon M \to M$
whose image is contained in $\mu\inv(V)$,
whose restriction to some neighbourhood of $\mu\inv(B)$
is the identity map on that neighbourhood,
and such that $\varphi$ is $G$-equivariantly smoothly homotopic
to the identity map on $M$. 
\end{Lemma}

\begin{proof}
Let $\eps_1$ and $\eps_3$ be, respectively, 
the radii of the balls $B$ and $V$.
Let $\eps_2$ be strictly between $\eps_1$ and $\eps_3$,
and let $V'$ be the open ball of radius $\eps_2$ in $\g^*$ that is centred
at the origin.

There exists a vector field on $[0,\infty)$ 
whose flow $(t,x) \mapsto \rho_t(x)$ is defined for all $t \geq 0$
and satisfies $\rho_t|_{[0,\eps_{2}]} 
                = \text{Identity}|_{[0,\eps_{2}]}$
for all $t$ and $\rho_1([0,\infty)) \subset [0,\eps_3)$ when $t=1$.
The flow
$$ \rho^\calT_t(\beta) := \begin{cases}
 \rho_t(|\beta|) \cdot \frac{\beta}{|\beta|} & \beta \neq 0 \\
 0 & \beta = 0
\end{cases}$$
on $\calT$ is defined for all $t \geq 0$
and satisfies $\rho^\calT_t|_{V'} = \text{Identity}|_{V'}$ for all $t$
and $\rho_1(\calT) \subset V$ when $t=1$.

Let $X^\calT$ be the vector field on $\calT$ 
that generates the flow $\rho^\calT$.
The assumptions on $\rho^\calT$ imply that $X^\calT$ is a multiple
of the Euler vector field on $\g^*$
by a smooth $\Ad^*$-invariant function $h \colon \calT \to \R$
that vanishes on $B$.  
Take a vector field on $M$ that lifts the Euler vector field on $\g^*$;
by averaging, we may assume that it is $G$ invariant;
let $X^M$ be its multiple by the function $h \circ \mu$.
Then $X^M$ is a $G$-invariant vector field
on $M$ that vanishes on $\mu\inv(B)$ and the map $\mu$
intertwines $X^M$ with $X^\calT$.
By Lemma \ref{flow lifts}, 
the flow $\rho^M_t$ of the vector field $X^M$ is defined for all $t \geq 0$
and lifts the flow $\rho^\calT_t$.
The time one map $\varphi := \rho^M_1$ has the required properties.
\end{proof}

The next lemma will be used in the proof of Proposition \ref{prop:local global}.

\begin{Lemma} \labell{meets B and extend oneform}
Let $(M,\omega,\mu)$ be an exact Hamiltonian $G$ manifold.
Let $\calT$ be an open subset of $\g^*$
that is $\Ad^*$-invariant and is starshaped about the origin. 
Suppose that the image of $\mu$ is contained in $\calT$
and that $\mu$ is proper as a map to $\calT$.
Let $B$ be a closed ball in $\g^*$ centred at the origin
and contained in $\calT$.  Then 
\begin{enumerate}
\item Every connected component of $M$ meets $\mu^{-1}(B)$.
\item Let $V$ be an open ball in $\calT$ 
that is centred at the origin and that contains $B$.
Let $\hlambda$ be a $G$ invariant one-form on $\mu\inv(V)$
such that $d\hlambda = \omega$.
Then there exists a $G$ invariant one-form $\lambda$ on $M$
such that $d\lambda = \omega$ and such that
$\lambda = \hlambda$ on $\mu\inv(B)$.
\end{enumerate}
\end{Lemma}

\begin{proof}
By Lemma~\ref{lem:liouville}, the Liouville vector field on $M$
lifts the Euler vector field on $\g^*$.
Part (1) then follows from Lemma~\ref{retraction}.

Let $\lambda'$ be a $G$ invariant one-form on $M$
such that $d\lambda' = \omega$; such a one-form exists
because $\omega$ is exact.
The difference $\hlambda - \lambda'$ 
is a closed one-form, defined on $\mu\inv(V)$.
By Lemma~\ref{lem:liouville},
the Liouville vector field on $M$ lifts the Euler vector field on $\g^*$.
By Lemma~\ref{retraction},
there exists a $G$-equivariant smooth map $\varphi \colon M \to M$ 
whose image is contained in $\mu\inv(V)$
and such that $\varphi|_{\mu\inv(B)} = \text{identity}_{\mu\inv(B)}$.
Then
$$\lambda := \lambda' + \varphi^* ( \hlambda - \lambda') $$
is defined on all of $M$.
Because $d\lambda' = \omega$ and $d(\hlambda-\lambda')= 0$,
we have $d\lambda = \omega$.
Because $\varphi|_{\mu\inv(B)} = \text{Identity}_{\mu\inv(B)}$,
we have $\lambda = \hlambda$ on $\mu\inv(B)$.
This completes the proof of Part~(2).
\end{proof}

\section{Proof of the main result}
\labell{sec:liouville}

In this section we finally prove Theorem \ref{main}. 
The proof is based on Theorem \ref{main to starshaped} below.

Recall that an isomorphism of Hamiltonian $G$ manifolds 
is an equivariant symplectomorphism that intertwines the momentum maps. 
We call two Hamiltonian $G$ manifolds \textbf{isomorphic} 
if there is an isomorphism between them.
The \textbf{restriction} 
of a Hamiltonian $G$ manifold $(M,\omega,\mu)$
to an $\Ad^*$-invariant open subset $\calT \subset \g^*$ is
$\big( \mu^{-1}(\calT), \omega|_{\mu\inv(\calT)}, \mu|_{\mu\inv(\calT)} \big)$. 

The following lemma is used in the proof 
of Theorem \ref{main to starshaped}.

\begin{Lemma}[Local isomorphism] \labell{le:local}
Let $G$ be a compact Lie group
and let $(M,\omega,\mu)$ be an exact Hamiltonian $G$ manifold.
Suppose that $M/G$ is connected.
Let $\calT \subset \g^*$ be an $\Ad^*$-invariant open subset
that is starshaped about the origin.
Suppose that the image of $\mu$ is contained in $\calT$
and the map $\mu \colon M \to \calT$ is proper.
Then there exists an $\Ad^*$-invariant open neighbourhood $W$ 
of the origin $0 \in \g^*$
and a finite disjoint union 
of centred Hamiltonian $G$ models, $\bigsqcup(Y_j,\omega_{Y_j},\mu_{Y_j})$, 
whose restriction to $W$ is isomorphic to the restriction 
of $(M,\omega,\mu)$ to $W$.
\end{Lemma}

\begin{proof}
Since $\mu \colon M\to \calT$ is proper, the level set $\mu^{-1}(0)$ 
is compact. 
By Proposition \ref{sjamaar}, $\mu^{-1}(0)$ is a finite union 
of $G$ orbits, $\mu^{-1}(0) = \bigsqcup_{j=1}^N G \cdot x_j$.

Every orbit in the zero level set of the momentum map is isotropic.  
(More generally, in any Hamiltonian $G$ manifold $(M',\omega',\mu')$,
for any orbit $\calO \subset M'$, its image is 
a coadjoint orbit $\mu(\calO) \subset \g^*$,
and the pullback of the Kirillov-Kostant-Souriau form
under the momentum map $\mu'|_\calO \colon \calO \to \mu'(\calO)$
coincides with the pullback of the ambient symplectic form $\omega'$
by the inclusion map $\calO \hookrightarrow M'$.
See Kazhdan-Kostant-Sternberg \cite{kks}.)

The local normal form theorem for isotropic orbits (see
Theorem~\ref{lnf isotropic}) implies that, for each~$j$,
there exist a centred Hamiltonian $G$ model $(Y_j,\omega_{Y_j},\mu_{Y_j})$, 
a $G$-invariant neighbourhood $U'_j$ of the central orbit in $Y_j$,
and a $G$-invariant neighbourhood $U_j$ of $G \cdot x_j$ in $M$,
such that $\big( U_j,\omega|_{U_j},\mu|_{U_j} \big)$
is isomorphic to $\big( U'_j, {\omega_{Y_j}}|_{U'_j},{\mu_{Y_j}}|_{U'_j} \big)$.

In particular, each $U'_j \cap \mu_{Y_j}\inv(0)$ consists of a single orbit.
By Lemma \ref{muY is proper}, for each $j$,
the level set $\mu_{Y_j}\inv(0)$ consists of only the central orbit,
and the momentum map $\mu_{Y_j} \colon Y_j \to \g^*$ is proper.
Because $\bigsqcup_j \mu_{Y_j} \colon \bigsqcup_j Y_j \to \g^*$ 
and $\mu \colon M \to \g^*$ are proper,
there exists an $\Ad^*$-invariant neighbourhood $W$ of $0$ in $\g^*$ 
whose preimage under $\bigsqcup_j \mu_{Y_j}$ is contained 
in the neighbourhood $\bigsqcup_j U'_j$ of $\bigsqcup_j \mu_{Y_j}^{-1}(0)$
and whose preimage under $\mu$ is contained 
in the neighbourhood $U$ of $\mu^{-1}(0)$.
The restrictions 
of $(M,\omega,\mu)$ and $\bigsqcup_j (Y_j,\omega_{Y_j},\mu_{Y_j})$ to $W$
are isomorphic. This proves Lemma \ref{le:local}.
\end{proof}

The following lemma will be used in the proof 
of Proposition \ref{prop:local global}.

\begin{Lemma} \labell{liouville}
Let $G$ be a compact connected Lie group,
and let $(M,\omega,\mu)$ an exact Hamiltonian $G$-manifold. 
Let $\calT \subset \g^*$ be an $\Ad^*$-invariant open subset 
that is starshaped about the origin. 
Assume that the image of $\mu$ is contained in $\calT$, 
and that the map $\mu \colon M \to \calT$ is proper.
Let $\lambda$ be a $G$ invariant one-form on $M$
such that $d\lambda = \omega$ and $\mu^\xi = \iota(\xi_M)\lambda$
for all $\xi \in \g$.
Let $X$ be the vector field on $M$ that satisfies
$$ \iota(X) \omega = - \lambda ; $$
this is the negative of the Liouville vector field corresponding to $\lambda$.
Let $\Psi_t$ be the flow of the vector field $X$.  Then 
\begin{itemize}
\item
The flow $\Psi_t$ is defined for all $t \geq 0$.
\item
The flow $\Psi_t$ is $G$ equivariant.
\item
We have $\Psi_t^* \lambda = e^{-t} \lambda$ for all $t \geq 0$.
\item
For each neighbourhood $B'$ of the origin $0$ in $\g^*$,
and for each $t \geq 0$, 
the flow $\Psi_t$ restricts to a diffeomorphism
$$ \Psi_t \colon \mu\inv(e^t B' \cap \calT) \to \mu\inv(B' \cap e^{-t} \calT) $$
with inverse
$$ \Psi_{-t} \colon 
       \mu\inv(B' \cap e^{-t} \calT) \to \mu\inv(e^t B' \cap \calT) .$$
\end{itemize}
\end{Lemma}

\begin{proof}
By Lemma \ref{lem:liouville},
$\mu$ intertwines $X$ with the negative of the Euler vector field on $\g^*$.
Because $\calT$ is starshaped about the origin,
the flow of the restriction to $\calT$ of the Euler vector field 
is defined for all $t \geq 0$.
Because $\mu \colon M \to \calT$ is proper and by Lemma \ref{flow lifts}, 
we deduce that the flow of $X$ is also defined for all $t \geq 0$.

By Lemma \ref{lem:liouville}, $X$ is $G$ invariant.
This implies that $\Psi_t$ is $G$ equivariant.

By Lemma \ref{lem:liouville}, $\LL_X \lambda = - \lambda $.
This implies that $\Psi_t^*\lambda = e^{-t}\lambda$.

The last claim follows from Lemma \ref{flow lifts}
and the fact that multiplication by $e^{-t}$
restricts to a diffeomorphism from $e^t B' \cap \calT$ 
to $B' \cap e^{-t} \calT$.
\end{proof}

The following proposition will be used 
in the proof of Theorem \ref{main to starshaped}.

\begin{Proposition}[Local to global] \labell{prop:local global} 
Let $G$ be a compact Lie group,
and let $(M,\omega,\mu)$ and $(M',\omega',\mu')$ be 
exact Hamiltonian $G$-manifolds. 
Let $\calT \subset \g^*$ be an $\Ad^*$-invariant
open subset 
that is starshaped about the origin. 
Assume that the images of $\mu$ and $\mu'$ are contained in $\calT$, 
and that the maps $\mu \colon M \to \calT$ and $\mu' \colon M' \to \calT$ 
are proper.
Finally, assume that there exists
an $\Ad^*$-invariant neighbourhood $W$ of $0$ 
such that the restrictions of $(M,\omega,\mu)$ and $(M',\omega',\mu')$ 
to $W$ are isomorphic. 
Then $(M,\omega,\mu)$ and $(M',\omega',\mu')$ are isomorphic.
\end{Proposition}

\begin{proof}

By assumption, there exist a neighbourhood $W$ of the origin in $\g^*$
and an isomorphism $F_0 \colon \mu\inv(W) \to {\mu'}\inv(W)$
between the restrictions of $(M,\omega,\mu)$ and $(M',\omega',\mu')$ to $W$.

Fix an $\Ad^*$-invariant inner product on $\g^*$.
Without loss of generality,
assume that the neighbourhood $W$ 
is a ball in $\g^*$ that is centred at the origin.

Choose a $G$-invariant primitive $\lambda'$ of $\omega'$ 
that gives rise to the momentum map $\mu'$:
\begin{equation} \labell{properties of lambda prime}
 d\lambda' = \omega' \quad \text{ and } \quad 
 \iota(\xi_{M'}) \lambda' = {\mu'}^\xi \ \text{ for all } \xi \in \g . 
\end{equation}

The pullback
\begin{equation} \labell{eq:lam 0}
 \hlambda := F_0^* \lambda'
\end{equation}
is a $G$-invariant one-form on $\mu^{-1} (W)$, satisfying 
\begin{equation} \labell{properties of hlambda}
   d\hlambda = \omega|_{\mu^{-1}(W)} \ , \ \text{ and } \ 
    \iota(\xi_M) \hlambda = \mu^\xi|_{\mu^{-1}(W)} 
    \ \text{ for all } \xi \in \g .
\end{equation}

Let $B$ be a closed ball that is centred at the origin and is contained in $W$.
By Part~(2) of Lemma \ref{meets B and extend oneform}, 
we may find a $G$-invariant one-form $\lambda$ on $M$
such that $d\lambda = \omega$ and such that $\lambda$ and $\hlambda$ 
coincide on $\mu\inv(B)$.
By this and~\eqref{properties of hlambda}, 
$$ \iota(\xi_M) \lambda|_{\mu^{-1}(B)}  = \mu^\xi|_{\mu^{-1}(B)} 
   \ \text{ for all } \xi \in \g. $$
By Part (1) of Lemma~\ref{meets B and extend oneform}
we deduce that the momentum map that is obtained from $\lambda$ 
coincides with $\mu$.
So 
\begin{equation} \labell{properties of lambda}
   d\lambda = \omega    \quad \text{ and } \quad
   \iota(\xi_M) \lambda = \mu^\xi \ \text{ for all } \xi \in \g .
\end{equation}

We define $X$ to be the unique vector field on $M$ satisfying
$ \iota(X) \omega = - \lambda $ 
and take $\Psi_t$ to be its flow.
Similarly,
we define $X'$ to be the unique vector field on $M'$ satisfying
$ \iota(X') \omega' = - \lambda' $ 
and take $\Psi'_t$ to be its flow.
By Lemma \ref{liouville}, these flows are defined for all $t \geq 0$.

Let $B'$ be the interior of $B$.
The restriction of $F_0$ to $\mu\inv(B')$
is a $G$ equivariant diffeomorphism
$$ F_B \colon \mu\inv(B') \to {\mu'}\inv(B')$$
that satisfies $F_B^*\lambda' = \lambda$;
this implies that $F_B$ intertwines $X$ with $X'$
and hence $\Psi_t$ with $\Psi'_t$.

For every $t \geq 0$,
define $\tF_t$ 
by the requirement that the following diagram commute.
$$ \begin{CD}
\mu\inv(e^t B' \cap \calT) @> \tilde{F}_t >> {\mu'}\inv(e^t B' \cap \calT) \\
@V \Psi_t VV @VV \Psi'_t V  \\
\mu\inv( B' \cap e^{-t} \calT ) @> F_B >> {\mu'}\inv( B' \cap e^{-t}\calT) .
\end{CD} $$
By Lemma \ref{liouville},
the vertical arrows in this diagram
are $G$ equivariant diffeomorphisms, and they satisfy
$\Psi_t^* \lambda = e^{-t} \lambda$
and
${\Psi'_t}^* \lambda' = e^{-t} \lambda'$.
The bottom arrow is also a $G$ equivariant diffeomorphism,
and it satisfies $F_B^* \lambda' = \lambda$.
These facts imply that $\tF_t$ is an equivariant diffeomorphism
and that $\tF_t^* \lambda' = \lambda$.
By \eqref{properties of lambda}
and \eqref{properties of lambda prime},
$\tF_t$ is an isomorphism of Hamiltonian $G$ manifolds.

For all $t \geq 0$ and $s \geq 0$ we have
\begin{align*}
   \Psi'_{t+s} \circ \tF_t
   & = \Psi'_s \circ \Psi'_t \circ \tF_t 
                     \qquad \text{because $\Psi'$ is a flow} \\
   & = \Psi'_s \circ F_B \circ \Psi_t 
                     \qquad \text{from the definition of $\tF_t$} \\
   & = F_B \circ \Psi_s \circ \Psi_t
                     \qquad \text{because $F_B$ intertwines the flows} \\
   & = F_B \circ \Psi_{t+s} 
                     \qquad \text{because $\Psi$ is a flow}.
\end{align*}
By the definition of $\tF_{t+s}$,
this implies that $\tF_{t+s}|_{ e^t B' \cap \calT} = \tF_t $.

So the union over $t \geq 0$ of the maps $\tF_t$
gives a map $\tF \colon M \to M'$.
Because $\tF_t$ is a isomorphism for each $t$,
we have that $\tF$ is also an isomorphism.

\end{proof}

\begin{Theorem} [Isomorphism] \labell{main to starshaped} 
Let $G$ be a compact Lie group
and let $(M, \omega, \mu)$ be an exact Hamiltonian $G$ manifold.
Suppose that $M/G$ is connected.
Let $\calT \subset \g^*$ be an $\Ad^*$-invariant open subset 
that is starshaped about the origin.
Suppose that the image of $\mu$ is contained in $\calT$ and the map
$\mu \colon M \to \calT$ is proper. 
Then there exists a centred Hamiltonian $G$ model $(Y, \omega_Y, \mu_Y)$ 
whose restriction to $\calT$ is isomorphic to $(M,\omega,\mu)$.
\end{Theorem}

\begin{proof}
By Lemma~\ref{le:local} and Proposition \ref{prop:local global},
there exists a finite disjoint union of centred Hamiltonian $G$ models,
$\bigsqcup_j (Y_j , \omega_{Y_j} , \mu_{Y_j} )$,
whose restriction to $\calT$ is isomorphic to $(M,\omega,\mu)$.
If $M/G$ is connected, 
since each $(Y_j|_{\calT})/G$ is non-empty,
there can be only one $Y_j$ in this union.
This prove Theorem~\ref{main to starshaped}.
\end{proof}

Finally, we prove Theorem \ref{main}.

\begin{proof}[Proof of Theorem~\ref{main}] \ 
Let a compact Lie group $G$
act on a symplectic manifold $(M,\mu)$
with momentum map $\mu \colon M \to \g^*$.
Assume that $M/G$ is connected, $\mu$ is proper, and $\omega$ is exact.

Let $\mu'$ be the momentum map that is obtained from some invariant primitive 
of $\omega$.
Then $(M,\omega,\mu')$ is an exact Hamiltonian $G$ manifold.
By Theorem~\ref{main to starshaped},
there exists a centred Hamiltonian $G$ model $(Y,\omega_Y,\mu_Y)$
whose restriction to $\calT$ is isomorphic to $(M,\omega,\mu')$,
hence equivariantly sympectomorphic to $(M,\omega)$.
\end{proof}


{
\section{Properties of centred Hamiltonian $G$ models}
\labell{sec:properties}

In this section we present properties of centred Hamiltonian
$G$ models that we used to prove 
Corollaries~\ref{cor:properties}, \ref{cor:S1}, and~\ref{cor:contractible}
of Theorem~\ref{main}.

\bigskip

Let $G$ be a compact Lie group.
Let $(Y,\omega_Y,\mu_Y)$ be a centred Hamiltonian $G$-model,
associated with a closed subgroup $H$ of $G$
and a linear action of $H$ on a symplectic vector space $V$ 
with a quadratic momentum map $\mu_V \colon V \to \h^*$.

\bigskip

We recall from Lemma~\ref{alternative model}
that the model $Y$ can be identified with the vector bundle
\begin{equation} \labell{Y prime}
 G \times_H \left( \h^0 \times V \right) 
\end{equation}
over $G/H$, obtained as the quotient 
of $ G \times \left( \h^0 \times V \right) $
by the anti-diagonal $H$ action,
such that the embedding $v \mapsto [1,0,v]$ of $V$ 
as a subspace of the fibre is symplectic, 
the zero section $G \times_H \{ 0 \}$ is isotropic,
and the momentum map on the model is
\[[a,\nu,v] \mapsto \Ad^*(a) (\nu + \mu_V(v)),\]
where $\h^0$ is the annihilator of $\h$ in $\g^*$,
and where we identify $\h^*$ with the orthogonal complement of $\h^0$
in $\g^*$ with respect to some $\Ad$-invariant inner product.

\begin{Lemma} \labell{homotopy} 
The model $Y$ is homotopy equivalent to $G/H$,
and the Euler characteristic of $Y$ is non-negative.  
\end{Lemma}

\begin{proof}
In the model~\eqref{Y prime},
the map $[a,\nu,v] \mapsto [a,t\nu,tv]$, for $t \in [0,1]$,
gives a deformation retraction of $Y$ to the zero section.
The zero section, in turn, is naturally identified with $G/H$.
This proves the first sentence. 
The second sentence follows from the first because,
by H.~Hopf and H.~Samelson \cite[page 241]{HS},
the Euler characteristic of $G/H$ is non-negative.
\end{proof}

\begin{Lemma} \labell{fixed set connected}
The fixed point set in the model $Y$ is connected.
Moreover,
if the momentum map $Y \to \g^*$ is proper, 
then this fixed point set is either empty or contains exactly one point.
\end{Lemma}

\begin{proof}
Consider the model as described in~\eqref{Y prime}.
If $H \neq G$, then the fixed point set is empty.
Assume $H = G$.  Then the model is isomorphic to the vector space $V$
with the linear action of $H$ and with the momentum map $\mu_V$.
The fixed point set $V^H$ is connected because it is a linear subspace of~$V$.
Now assume that the momentum map $\mu$ is proper.
Then its zero level set $\mu\inv( \{ 0 \} )$ is compact.
The fixed point set $V^H$ then contains just one point,
because it is contained in $\mu\inv(\{ 0 \})$,
and a vector subspace of $V$ that is contained in a compact subset of $V$
must be trivial.
\end{proof}

Recall that the \emph{Gromov width} of a $2n$-dimensional symplectic manifold
$(M,\om)$ is defined to be the number
\[ \sup\big \{\pi r^2\,\big|\,B^{2n}_r
   \textrm{ symplectically embeds into }M\big\},\]
where $B^{2n}_r\subset\R^{2n}$ denotes the open ball of radius $r$ 
centred at the origin, equipped 
with the standard symplectic structure.
\begin{Lemma} \labell{Gromov width} 
The Gromov width of $Y$ is infinite.
\end{Lemma}

\begin{proof}
Recall that the model $Y$ is obtained as a symplectic reduction
of $T^*G \times V$, where the two-form and momentum map 
on $T^*G \times V$ are induced from the one-form 
$\lambda_\taut \oplus \lambda_V$
that is described in Section~\ref{sec:models}.
Define a flow $\rho_t$ on $T^*G \times V$ by fibrewise multiplication
by $e^t$ on $T^*G$ and multiplication by $e^{\frac{1}{2}t}$ on $V$.
Then 
$\rho_t^* (\lambda_\taut \oplus \lambda_V) 
 = e^t (\lambda_\taut \oplus \lambda_V)$,
and $\rho_t$ commutes with the anti-diagonal $H$ action
by which we quotient to get $Y$.
It follows that $\rho_t$ descends to the symplectic reduction
$Y = (T^*G \times V) /\!/ H$
and satisfies $\rho_t^* \lambda_Y = e^t \lambda_Y$,
hence $\rho_t^* \omega_Y = e^t \omega_Y$.

By Darboux's theorem, for some $\eps > 0$ 
there exists a symplectic embedding
$i \colon B^{2n}_\eps \to Y$.
The composition
$ x \mapsto \rho_t(i( e^{-\frac{1}{2}t} x)) $
is a symplectic embedding of $B^{2n}_r$ into $Y$
where $r = \eps e^{\frac{1}{2}t}$. 
Because $\pi r^2 = \pi \eps^2 e^t$ can be made arbitrarily large
by appropriate choices of $t$, the Gromov width of $Y$ is infinite.
\end{proof}

We now present two situations in which we get even more precise information
on the models that can occur in Theorem \ref{main}.

\begin{Lemma}[Circle actions] \labell{circle}
Suppose that $G$ is the circle group $S^1$
and that its action on $Y$ is faithful.
Then $Y$ is equivariantly symplectomorphic to one of the following examples.
\begin{itemize}
\item[(a)]
The cylinder $S^1 \times \R$,
with the standard symplectic form $dq \wedge dp$
where $q \mod \Z$ is a coordinate on $S^1 \cong \R/\Z$ 
and $p$ is a coordinate on $\R$.
The circle $S^1$ acts by rotations on the $S^1$ factor.
\item[(b)]
The vector space $\C^n$, with the standard symplectic form.
The circle $S^1\sub\C$ acts by 
$ \lambda \cdot (z_1,\ldots,z_n) = 
  (\lambda^{m_1} z_1, \ldots, \lambda^{m_n} z_n) $,
where $m_1,\ldots,m_n$ are either all positive integers 
or all negative integers.
\end{itemize}
\end{Lemma}

\begin{proof}
Because the $S^1$ action is faithful on $Y$,
the $H$ action is faithful on $V$.
(This uses the fact that $S^1$ is abelian.)
Because $\mu_Y \colon Y \to \g^*$ is proper,
the momentum map $\mu_V \colon V \to \h^*$ is proper. 
These two facts imply that either $H = \{ 1 \}$ and $V = \{ 0 \}$
or $H = S^1$ and $V \cong \C^n$
where $H$ acts on all the coordinates with positive weights
or on all the coordinates with negative weights.
\end{proof}

\begin{Lemma} \labell{contractible model}
Suppose that the model $Y$ is contractible.  Then $(Y,\omega_Y,\mu_Y)$ 
is equivariantly symplectomorphic to $(V,\omega_V,\mu_V)$. 
\end{Lemma}

\begin{proof}
By Lemma \ref{homotopy}, since $Y$ is contractible, so is $G/H$.  
Because $G$ is a compact connected Lie group and $H$ a closed subgroup,
the quotient $G/H$ is a closed manifold.
Being a contractible closed manifold, $G/H$ is a point.
So $H=G$. 
It now follows from the description~\eqref{Y prime}
that $Y$ is isomorphic to $V$.
\end{proof}

}

\section{Exotic actions on symplectic vector spaces}
\labell{sec:nonproper}

If we drop the assumption that the momentum map be proper,
the conclusions of Theorem~\ref{main}
and of its consequences can fail dramatically, 
even when the symplectic manifold is a symplectic vector space.
We show this 
through examples of exotic Hamiltonian actions on symplectic vector spaces.
Corollaries~\ref{F fixed}, \ref{not linearizable}, and~\ref{bad U revised},
give examples where, respectively, the conclusions 
of Corollary \ref{cor:properties} part (ii),
of Corollary \ref{cor:contractible},
and of Theorem \ref{main}, fail.

Our main ingredient is the following result,
which was communicated to us by K.~Pawa\l owski:

\begin{Proposition} \labell{prop:fixed} 
Let $G$ be a compact connected nonabelian Lie group.
Then there exists a smooth $G$ action on a Euclidean space
without fixed points.
Moreover, let $F$ be a closed, stably parallelizable manifold;
then there exists a smooth $G$ action on a Euclidean space
whose fixed point set is diffeomorphic to $F$.
\end{Proposition}

Recall that a manifold is called stably parallelizable 
if the direct sum of its tangent bundle with some trivial vector bundle 
is trivializable.
Examples include the empty set, finite sets, Lie groups, 
spheres of any dimension, and toric manifolds.

\begin{proof}[Proof of Proposition~\ref{prop:fixed}]
Let $Y$ be a finite dimensional countable $G$--CW complex,
with finitely many orbit types and with no fixed points,
which is contractible.

The construction of such a $Y$ is described by Wu-chung Hsiang and Wu-yi Hsiang 
in \cite[Theorem 1.9]{hsiang-hsiang}
and is based on earlier ideas of Conner, Floyd, and Montgomery 
\cite{conner-floyd,conner-montgomery}.
The idea is to take
a representation $V$ of $G$ without a trivial summand
and an equivariant map $f \colon S(V) \to S(V)$ of degree zero
on its unit sphere,
as constructed in \cite[Proposition 1.4]{hsiang-hsiang},
and to take $Y$ to be the tower of mapping cylinders of $f$.

Let $X$ be the topological join of $Y$ with $F$.
Then $X$ is a contractible finite dimensional countable $G$--CW complex
with finitely many orbit types and with fixed point set $F$.

Let $B$ be the union of $F$ with the equivariant $0$--cells in $X \ssminus F$.
Let $m$ be such that the direct sum of $TF$ with a trivial vector bundle 
is isomorphic to the trivial bundle $F \times \R^m$,
and let $E = X \times \R^m \to X$ be the trivial bundle.
We now apply a theorem of K.~Pawa\l owski \cite[Theorem~3.1]{PaFixed},
which is based on Pawa\l owski's equivariant thickening procedure
\cite[Theorem 2.4]{PaFixed}.
Given these $X$, $F$, $B$, and $E$, 
this theorem yields a smooth $G$--manifold $M$
that contains $B$ as a smooth invariant submanifold and such that $M^G = F$, 
and a homotopy equivalence $M \to X$.
As noted in \cite[Remarks~2.5 and~3.2]{PaFixed}, 
we can assume that $\dim M \geq 5$,
and the construction yields an $M$ that is contractible
and simply connected at infinity;
by a result of Stallings \cite[Corollary~5--1]{Sta},
$M$ is diffeomorphic to a Euclidean space.
\end{proof}

\begin{Corollary} \labell{F fixed} 
Let $G$ be a compact connected nonabelian Lie group.
Then there exists a Hamiltonian $G$ action on a symplectic vector space
without fixed points.
Moreover, let $F$ be a closed, stably parallelizable manifold;
then there exists a Hamiltonian $G$ action on a symplectic vector space
whose fixed point set is symplectomorphic to $T^*F$. 
\end{Corollary}

\begin{proof}
By Proposition \ref{prop:fixed} there exists 
a smooth $G$ action on a Euclidean space $X$ 
whose fixed point set is diffeomorphic to $F$.
The fixed point set of the induced $G$ action on the cotangent bundle $T^*X$
is symplectomorphic to the cotangent bundle $T^*F$ of the fixed point set $F$;
this is a consequence of the slice theorem for compact group actions.
We obtain Corollary~\ref{F fixed}
by taking $V$ to be the direct sum $X \oplus X^*$,
identified with the cotangent bundle $T^*X$,
with the cotangent lifted $G$ action.
\end{proof}

\begin{Corollary} \labell{not linearizable}
Let $G$ be a compact connected nonabelian Lie group.
Then there exists a Hamiltonian $G$ action on a symplectic vector space 
that is not isomorphic to a linear action.
\end{Corollary}

\begin{proof}
In Corollary~\ref{F fixed}, take an action without fixed points.
\end{proof}

\begin{Corollary} \labell{bad U revised}
Let $G$ be a compact connected nonabelian Lie group.
Then there exists a Hamiltonian $G$ action 
on a symplectic vector space $(V,\omega_V)$
with the following properties.
\begin{itemize}
\item[(i)]
There exist two points in $V$,
such that no $G$ invariant subset of $V$ that contains both points
admits an equivariant open embedding into any centred Hamiltonian $G$ model.
\item[(ii)]
There exists a connected bounded $G$ invariant open subset of $V$ 
that does not admit an equivariant open embedding 
into any centred Hamiltonian $G$ model.
\end{itemize}
\end{Corollary}

\begin{proof}
By Corollary \ref{F fixed}, there exists a $G$ action 
on a symplectic vector space $(V,\omega_V)$ with exactly two fixed points.
By Lemma~\ref{fixed set connected},
the fixed point set of any centred Hamiltonian $G$ model is connected.
For (i), note that a manifold with a $G$ action with two isolated
fixed points does not admit an equivariant open embedding
into any manifold with a $G$ action whose fixed point set is connected.
For (ii), take the invariant open subset to be $G \cdot B$ 
where $B$ is an open ball that contains the two fixed points.
\end{proof}

%
%
%

\section{Acknowledgments}

We would like to thank Krzysztof Pawa\l owski for explaining the proof
of Proposition \ref{prop:fixed} to us,
and we would like to thank Kai Cieliebak for a helpful discussion.
This research was partially funded by 
the Natural Sciences and Engineering Research Council of Canada.

\end{document}